\documentclass[a4paper,12pt,reqno]{amsart}

\usepackage{t1enc}
\usepackage[english]{babel}
\usepackage{amsmath,amssymb,eucal,amsthm}
\usepackage{bm}
\usepackage{graphicx}
\usepackage{mathrsfs}
\usepackage{mathptmx}
\usepackage{latexsym}
\usepackage{ulem}

\newtheorem{theorem}{Theorem}
\newtheorem{lemma}[theorem]{Lemma}

\theoremstyle{definition}
\newtheorem{remark}{Remark}

\newtheorem*{thm}{Theorem~{A}}
\newtheorem*{thmb}{Theorem~{B}}

\def\nn{\mathbb{N}}

\def\cc{\mathbb{C}}

\def\gb{\mathfrak{B}}
\def\n0{{\nn\cup\{0\}}}

\def\vp{\varphi}

\begin{document}
\hfill\texttt{\jobname.tex}\qquad\today

\bigskip
\title[Remarks on the mixed joint universality]
{Remarks on the mixed joint universality \\ for a class of zeta-functions}

\author{Roma Ka{\v c}inskait{\.e}}

\address{R. Ka{\v c}inskait{\.e} \\
Department of Mathematics, {\v S}iauliai University, Vi\-{\v s}ins\-kio~19, LT-77156 {\v S}iauliai, Lithuania\\
Department of Mathematics and Statistics, Vytautas Magnus University, Kaunas, Vileikos 8, LT-44404, Lithuania}
\email{r.kacinskaite@fm.su.lt, r.kacinskaite@if.vdu.lt}

\author{Kohji Matsumoto}

\address{K. Matsumoto, Graduate School of Mathematics, Nagoya University, Chikusa-ku,
Nagoya 464-8602, Japan}
\email{kohjimat@math.nagoya-u.ac.jp}
\date{}

\maketitle

\begin{abstract}
Two remarks related with
the mixed joint universality for a polynomial Euler pro\-duct $\vp(s)$ and a
periodic Hurwitz zeta-function $\zeta(s,\alpha;\gb)$, when $\alpha$ is a
transcendental pa\-ra\-me\-ter, are given.    One is the mixed joint functional independence,
and the other is a generalized universality, which includes several
periodic Hurwitz zeta-functions.
\end{abstract}



\section{Introduction}

In our former paper \cite{RK-KM-15}, we have shown a certain mixed joint universality theorem,
which is valid for an Euler product of rather general form, and a periodic Hurwitz
zeta-function.

In the present note, we give two remarks related with the result in \cite{RK-KM-15}.
The first remark is on the mixed joint functional independence result.
It is well known that functional independence properties can be deduced from universality
results.    We will show that such a functional independence is also valid in our present
mixed joint situation.

The second remark is on a generalization of the result in \cite{RK-KM-15}.
It is an important problem to study how general the mixed joint universality property holds.
We will prove a ge\-ne\-ra\-li\-zed limit theorem and a generalized
universality theorem, which involve several periodic Hurwitz
zeta-functions, under a certain matrix condition.

\section{Functional independence}

The history of the problem on the functional independence of Dirichlet series
goes back to the famous lecture of D.~Hilbert \cite{HilPr-69} in 1900.
He mentioned that the Riemann zeta-function $\zeta(s)$ does not satisfy any non-trivial
algebraic differential equation.

Recall the definition of $\zeta(s)$.
Let $\mathbb{P}$ be the set of all prime numbers, and $\mathbb{C}$ the set of all
complex numbers.
For $s=\sigma+it\in\mathbb{C}$, $\zeta(s)$ is given by
$$
\zeta(s)=\sum_{m=1}^{\infty}\frac{1}{m^s}=\prod_{p \in \mathbb{P}}\left(1-\frac{1}{p^s}\right)^{-1}
$$
for $\sigma>1$,
and can be analytically continued to the whole complex plane $\cc$ except for a simple pole
at the point $s=1$ with residue 1.

In 1973, S.M.~Vo\-ro\-nin proved \cite{SMV-73} the following functional independence
result on $\zeta(s)$.
Let $\mathbb{N}$ and $\mathbb{N}_0$ be the set of positive integers,
and non-negative integers, respectively.

\begin{thm}[\cite{SMV-73}]\label{Voronin-73}
Let $N\in\mathbb{N}$ and $n\in\mathbb{N}_0$.
The function $\zeta(s)$ does not satisfy any differential equation
$$
\sum_{j=0}^{n}s^j F_j\left(\zeta(s),\zeta'(s),\ldots,\zeta^{(N-1)}(s)\right)\equiv 0
$$
for continuous functions $F_j$, $j=0,\ldots,n$, not all identically zero.
\end{thm}

Later this result was generalized to other zeta- and $L$-functions. For the survey, see the monographs by A.~Laurin\v cikas \cite{AL-96} and J.~Steuding \cite{JSt-07}.

Nowadays in analytic number theory, the investigation of statistical properties
(and also the functional independence) for a collection of various zeta-functions,
some of them have the Euler product expansion and others do not have, is a very interesting problem since an important role is played by parameters included in the definition of
those functions.

The first result in this direction is due to H.~Mishou. In 2007, he proved \cite[Theorem~4]{HM-07} that the pair of zeta-functions consisting of the Riemann zeta-function $\zeta(s)$
and the Hurwitz zeta-function $\zeta(s,\alpha)$ is functionally independent.

We recall that the Hurwitz zeta-function $\zeta(s,\alpha)$ with a fixed parameter $\alpha$, $0<\alpha \leq 1$, is defined by the Dirichlet series
$$
\zeta(s,\alpha)=\sum_{m=0}^{\infty}\frac{1}{(m+\alpha)^s}
$$
for $\sigma>1$,
and can be continued to the whole complex plane, except, for a simple pole at the point $s=1$ with residue 1. In general, the function $\zeta(s,\alpha)$ has no Euler product over primes, except for the cases $\alpha=1$ and $\alpha=\frac{1}{2}$, when
$\zeta(s,\alpha)$ is essentially reduced to $\zeta(s)$.
Then Mishou's result is the following statement.

\begin{thmb}[\cite{HM-07}]\label{Mishou-2007}
Suppose that $\alpha$ is transcendental. Let $N \in \mathbb{N}$,
$n\in\mathbb{N}_0$, and
$F_j: \mathbb{C}^{2N}\to \mathbb{C}$ be a continuous function for each $j=0,\ldots,n$.
Suppose%
\begin{eqnarray*}
\sum_{j=0}^{n}s^j \cdot F_j\big(\zeta(s),\zeta'(s),\ldots,\zeta^{(N-1)}(s),\zeta(s,\alpha),\zeta'(s,\alpha),\ldots,\zeta^{(N-1)}(s,\alpha)\big)\equiv 0.
\end{eqnarray*}
Then $F_j\equiv 0$ for $j=0,\ldots,n$.
\end{thmb}

This is the first ``mixed joint'' functional independence theorem.
Later this result was generalized to the collection of a periodic zeta-function and
a periodic Hurwitz zeta-function by the first-named author and A.~Laurin\v cikas
in \cite{RK-AL-11}.

In this paper, we will prove rather general result on the mixed joint functional independence for a class of zeta-functions, consisting of the so-called Matsumoto zeta-functions and periodic Hurwitz zeta-functions.



Let ${\mathfrak{B}}=\{b_m: m \in \mathbb{N}_0\}$ be a periodic sequence of complex numbers
(not all zero)
with minimal period $k \in \mathbb{N}$, and $0<\alpha\leq 1$.
In 2006, A.~Javtokas and A.~Laurin\v cikas \cite{AJ-AL-06} introduced  the periodic Hurwitz zeta-function $\zeta(s,\alpha;{\mathfrak{B}})$. For $\sigma>1$, it  is given by the series
$$
\zeta(s,\alpha;{\mathfrak{B}})=\sum_{m=0}^{\infty}\frac{b_m}{(m+\alpha)^s}.
$$
It is known that
$$
\zeta(s,\alpha;{\mathfrak{B}})=\frac{1}{k^s}\sum_{l=0}^{k-1}b_l\zeta\left(s,\frac{l+\alpha}{k}\right), \quad \sigma>1.
$$
Therefore the function $\zeta(s,\alpha;{\mathfrak{B}})$ can be analytically continued
to the whole complex plane, except for a possible simple pole at the point $s=1$
with residue
$$
b:=\frac{1}{k}\sum_{l=0}^{k-1}b_l.
$$
If $b=0$, the corresponding periodic Hurwitz zeta-function is an entire function.

The functional independence of periodic Hurwitz zeta-functions was proved by
A.~Laurin\v cikas in \cite[Theorem~1]{AL-08}.

Now we recall the definition of the polynomial Euler products ${\widetilde{\varphi}}(s)$, or, so-called Matsumoto zeta-functions.

Let, for $m \in \mathbb{N}$, $g(m)\in \mathbb{N}$, and, for $j \in \mathbb{N}$ and
$1\leq j \leq g(m)$, $f(j,m)\in \mathbb{N}$. Denote by $p_m$ the $m$th prime number, and $a_m^{(j)}\in \mathbb{C}$. The zeta-function ${\widetilde{\varphi}}$ was introduced by the second-named author in \cite{KM-90}, and it is defined by the polynomial Euler product
\begin{equation}\label{km-eq-1}
{\widetilde\varphi}(s)=\prod_{m=1}^{\infty}\prod_{j=1}^{g(m)}\left(1-a_m^{(j)}p_m^{-sf(j,m)}\right)^{-1}.
\end{equation}

Suppose that
\begin{equation}\label{km-eq-2}
g(m)\leq C_1p_m^\alpha \quad \text{and} \quad  |a_m^{(j)}|\leq p_m^\beta
\end{equation}
with a positive constant $C_1$ and non-negative constants $\alpha$ and $\beta$.
In view of \eqref{km-eq-2}, the function $\widetilde{\varphi}(s)$ converges absolutely for $\sigma>\alpha+\beta+1$ (see Appendix), and hence in this region it can be expressed as
the Dirichlet series
\begin{align}\label{km-eq-2.5}
{\widetilde \varphi}(s)=\sum_{k=1}^{\infty}\frac{{\widetilde c}_k}{k^s},
\end{align}
with coefficients ${\widetilde c}_k$. For brevity, denote the shifted version of the function ${\widetilde \varphi}(s)$ by
\begin{equation}\label{km-eq-3}
\varphi(s)={\widetilde \varphi}(s+\alpha+\beta)=\sum_{k=1}^{\infty}\frac{{\widetilde c}_k}{k^{s+\alpha+\beta}}=\sum_{k=1}^{\infty}\frac{c_k}{k^s}
\end{equation}
with $c_k=k^{-\alpha-\beta}{\widetilde c}_k$. Then, for $\sigma>1$, the series on the right-hand side of \eqref{km-eq-3} converges absolutely.

The aim of this paper is to obtain a mixed joint functional independence of the collection
of zeta-functions consisting of the Matsumoto zeta-function $\varphi(s)$ belonging to
the Steuding subclass ${\widetilde S}$, defined below, and periodic Hurwitz zeta-functions
$\zeta(s,\alpha;{\mathfrak{B}})$.

We say that the function $\varphi(s)$ belongs to the class ${\widetilde S}$ if following conditions are satisfied:
\begin{itemize}
  \item[(i)] there exists a Dirichlet series expansion
  $$
  \varphi(s)=\sum_{m=1}^{\infty}\frac{a(m)}{m^s}
  $$
  with $a(m)=O(m^\varepsilon)$ for every $\varepsilon>0$;
  \item[(ii)] there exists $\sigma_\varphi<1$ such that $\varphi(s)$ can be meromorphically continued to the half-plane $\sigma>\sigma_\varphi$;
  \item[(iii)] there exists a constant $c \geq 0$ such that
  $$
  \varphi(\sigma+it)=O(|t|^{c+\varepsilon})
  $$
  for every fixed $\sigma>\sigma_\varphi$ and $\varepsilon>0$;
  \item[(iv)] there exists the Euler product expansion over prime numbers, i.e.,
  $$
  \varphi(s)=\prod_{p \in \mathbb{P}}\prod_{j=1}^{l}\left(1-\frac{a_j(p)}{p^s}\right)^{-1};
  $$
  \item[(v)] there exists a constant $\kappa>0$ such that
  $$
  \lim_{x \to \infty}\frac{1}{\pi(x)}\sum_{p \leq x}|a(p)|^2=\kappa,
  $$
  where $\pi(x)$ denotes the number of primes $p$, $p \leq x$.
\end{itemize}

This class was introduced by J.~Steuding in  \cite{JSt-07}, and is a subclass of the
class of Matsumoto zeta-functions.
For $\varphi\in{\widetilde S}$, let $\sigma^*$ be the infimum of all $\sigma_1$ for which
$$
\frac{1}{2T}\int_{-T}^T |\varphi(\sigma+it)|^2 dt\sim\sum_{m=1}^{\infty}
\frac{|a(m)|^2}{m^{2\sigma}}
$$
holds for any $\sigma\geq\sigma_1$.    Then it is known that $\frac{1}{2}\leq\sigma^*<1$
(see \cite[Theorem 2.4]{JSt-07}).

We state the first main result of this paper in the following theorem.

\begin{theorem}\label{km-thm-1}
Suppose that $\alpha$ is a transcendental number, and the function $\varphi(s)$ belongs
to the class $\widetilde S$. Let $N \in \mathbb{N}$, $n\in\mathbb{N}_0$,
and the function
$F_j: {\mathbb{C}}^{2N}\to \mathbb{C}$ be continuous for $j=0,1,\ldots,n$.   If
\begin{align*}
&\sum_{j=0}^{n}s^j\cdot F_j\big(\varphi(s), \varphi'(s),\ldots,\varphi^{(N-1)}(s), \zeta(s,\alpha;\mathfrak{B}),\zeta'(s,\alpha;\mathfrak{B}),\ldots,\zeta^{(N-1)}(s,\alpha;\mathfrak{B})
\big)
\end{align*}
is identically zero, then $F_j \equiv 0$ for $j=0,1,\ldots,n$.
\end{theorem}

\section{Proof of Theorem 1}

To the proof of the mixed joint functional independence for the functions $\varphi(s)$
and $\zeta(s,\alpha;{\mathfrak{B}})$, we need two further propositions: the mixed joint universality theorem and the so-called denseness lemma.

\subsection{Mixed joint universality of the functions $\varphi(s)$ and $\zeta(s,\alpha;{\mathfrak{B}})$}

The proof of Theo\-rem~\ref{km-thm-1} is based on the mixed joint universality theorem in
the Voronin sense for the functions $\varphi(s)$ and $\zeta(s,\alpha;{\mathfrak{B}})$.
It was obtained by the authors in \cite[Theorem~2.2]{RK-KM-15}. We will give the statement of this universality theorem as a lemma.
Let $\mathbb{R}$ be the set of real numbers, and
$D(a,b)=\{s\in\mathbb{C}:a<\sigma<b\}$ for any $a<b$.
For any compact subset $K\subset\mathbb{C}$, denote by $H^c(K)$ the set of all
$\mathbb{C}$-valued continuous functions defined on $K$, holomorphic in the interior of
$K$.   By $H_0^c(K)$ we mean the subset of $H^c(K)$ consisting of all elements which are
non-vanishing on $K$.
\begin{lemma}[\cite{RK-KM-15}]\label{km-lem-1}
Suppose that $\varphi\in {\widetilde S}$, and $\alpha$ is a transcendental number. Let
$K_1$ be a compact subset of $D(\sigma^*,1)$, and $K_2$ be a compact subset of
$D(\frac{1}{2},1)$, both with connected complements. Suppose that $f_1\in H_0^c(K_1)$
and $f_2\in H^c(K_2)$.   Then, for every $\varepsilon>0$,
\begin{align*}
\liminf\limits_{T \to \infty}\frac{1}{T}\mu\bigg\{\tau\in [0,T]: &\; \sup\limits_{s \in K_1}|\varphi(s+i\tau)-f_1(s)|<\varepsilon, \\ &\; \sup\limits_{s\in K_2}|\zeta(s+i\tau,\alpha;{\mathfrak{B}})-f_2(s)|<\varepsilon\bigg\}>0,
\end{align*}
where $\mu\{A\}$ denotes the Lebesgue measure of the measurable set $A \subset {\mathbb{R}}$.
\end{lemma}

Note that for the proof of Lemma~\ref{km-lem-1} we use  the joint mixed limit theorem in the sense of weakly convergent probability measures for the Matsumoto zeta-functions $\varphi(s)$ and the periodic Hurwitz zeta-function $\zeta(s,\alpha;{\mathfrak{B}})$. It was proved by the authors in \cite[Theo\-rem~2.1]{RK-KM-15}.

\subsection{A denseness lemma}

For the proof of Theorem~\ref{km-thm-1}, we need a denseness lemma.

Define the map $u: {\mathbb{R}} \to {\mathbb{C}}^{2N}$ by the formula
\begin{eqnarray*}
u(t)&=&\big(\varphi(\sigma+it),\varphi'(\sigma+it),\ldots,\varphi^{(N-1)}(\sigma+it),\cr
&& \quad \zeta(\sigma+it,\alpha;{\mathfrak{B}}),\zeta'(\sigma+it,\alpha;{\mathfrak{B}}),\ldots,\zeta^{(N-1)}(\sigma+it,\alpha;{\mathfrak{B}})\big)
\end{eqnarray*}
with $\frac{1}{2}<\sigma<1$.

\begin{lemma}\label{km-lem-2}
Suppose that $\alpha$ is transcendental. Then the image of ${\mathbb{R}}$ by $u$ is
dense in ${\mathbb{C}}^{2N}$.
\end{lemma}

\begin{proof}
We will give a sketch, since the proof follows in the same way as Lemma~13 from \cite{RK-AL-11} (see, also, \cite[Theorem~3]{HM-07}).

We can find a sequence $\{\tau_m: \tau_m\in{\mathbb{R}}\}$,
$\lim_{m\to\infty}\tau_m=\infty$, such that the inequalities
$$
|\varphi^{(j)}(\sigma+i\tau_m)-s_{1j}|<\frac{\varepsilon}{2N}
$$
and
$$
|\zeta^{(j)}(\sigma+i\tau_m,\alpha;{\mathfrak{B}})-s_{2j}|<\frac{\varepsilon}{2N}
$$
hold for every $\varepsilon>0$ and arbitrary complex numbers $s_{lj}$, $l=1,2$, $j=0,\ldots,N-1$.    To show this, we consider the polynomial
$$
p_{lN}(s)=\frac{s_{l,N-1}\cdot s^{N-1}}{(N-1)!}+\frac{s_{l,N-2}\cdot s^{N-2}}{(N-2)!}
+\ldots+\frac{s_{l0}}{0!}, \quad l=1,2.
$$
Then, for $j=0,\ldots,N-1$ and $l=1,2$, we have
$$
p_{lN}^{(j)}(0)=s_{lj}.
$$
Now, in view of Lemma~\ref{km-lem-1} and repeating analogous arguments as in the proof of Lemma~13 from \cite{RK-AL-11}, we can prove the existence of the above sequence
$\{\tau_m\}$ and obtain that the image of ${\mathbb{R}}$ by the map $u$
is dense in ${\mathbb{C}}^{2N}$.
\end{proof}

\subsection{Proof of Theorem \ref{km-thm-1}}

Now we are ready to complete the proof of Theorem \ref{km-thm-1}.
The essential idea is due to Voronin (see, for example, \cite{SMV-77}).
We first prove that $F_n\equiv 0$.

Suppose that $F_n \not\equiv 0$.   It follows that there exists a point
$${\underline a} =(s_{10},s_{11},\ldots,s_{1,N-1},s_{20},s_{21},\ldots,s_{2,N-1})\in {\mathbb{C}}^{2N}$$
such that $F_n({\underline a})\not =0$.   From the continuity of the function $F_n$,
we find a bounded domain $G \subset {\mathbb{C}}^{2N}$ such that ${\underline a} \in G$, and, for all ${\underline s \in G}$,
\begin{equation}\label{km-eq-4}
|F_n({\underline s})| \geq c>0.
\end{equation}
By Lemma~\ref{km-lem-2}, there exists a sequence $\{\tau_m: \tau_m \in {\mathbb{R}}\}$, $\lim_{m\to \infty}\tau_m=\infty$, such that
\begin{eqnarray*}
&& \big( \varphi(\sigma+it),\varphi'(\sigma+it), \ldots,\varphi^{(N-1)}(\sigma+it),\cr
&& \quad \zeta(\sigma+it,\alpha;{\mathfrak{B}}),\zeta'(\sigma+it,\alpha;{\mathfrak{B}}),\ldots,\zeta^{(N-1)}(\sigma+it,\alpha;{\mathfrak{B}})
\big)\in G.
\end{eqnarray*}
But this together with \eqref{km-eq-4} contradicts to the hypothesis of the theorem if
$\tau_m$ is sufficiently large.   Hence, $F_n\equiv 0$.

Similarly we can show $F_{n-1}\equiv 0,\ldots,F_0\equiv 0$, inductively.   The proof
is complete.

\section{A generalization}

The mixed joint universality and the mixed joint functional independence theorem
can be obtained in the following more general situation.

Suppose that $\alpha_j$ be a real number with $0<\alpha_j<1$,
$l(j)\in\mathbb{N}$, $j=1,\ldots,r$, and $\lambda=l(1)+\ldots+l(r)$.
For each $j$ and $l$, $1\leq j\leq r$, $1\leq l\leq l(j)$, let
${\mathfrak{B}}_{jl}=\{b_{mjl}\in\mathbb{C}: m \in \mathbb{N}_0\}$ be a periodic
sequence of complex numbers (not all zero) with the minimal period $k_{jl}$, and let
$\zeta(s,\alpha_j;{\mathfrak{B}}_{jl})$ be the corresponding periodic Hurwitz zeta-function.
Denote by $k_j$ the least common multiple of periods $k_{j1},\ldots,k_{jl(j)}$.
Let $B_j$ be the matrix consisting of coefficients $b_{mjl}$ from the periodic sequences
${\mathfrak{B}}_{jl}$, $l=1,\ldots,l(j)$, $m=1,\ldots,k_j$, i.e.,
$$
B_j=\begin{pmatrix}
b_{1j1}  & b_{1j2} & \ldots & b_{1jl(j)}\cr
b_{2j1}  & b_{2j2} & \ldots & b_{2jl(j)}\cr
\ldots & \ldots & \ldots  & \ldots \cr
b_{k_j j1}  & b_{k_j j2} & \ldots & b_{k_j jl(j)}\cr
\end{pmatrix}.
$$

The functional independence for the above collection of periodic Hurwitz zeta-func\-tions
was proved by A.~Laurin\v cikas in \cite[Theorem~3]{AL-08} under a certain matrix
condition.
The proof is based on the joint universality theorem among periodic Hurwitz
zeta-functions proved by J.~Steuding in \cite{JSt-03}.

For the proof of mixed joint functional independence, we may adopt the method de\-ve\-lo\-ped in series of works by A.~Laurin\v cikas and his colleagues (see, for example, \cite{JG-RM-SR-DS-10}, \cite{AL-DS-12}, \cite{VP-DS-14}). Then it is possible to obtain the following
generalization of Theorem~\ref{km-thm-1}.

\begin{theorem}\label{km-thm-2}
Suppose that $\alpha_1,\ldots,\alpha_r$ are algebraically independent over ${\mathbb{Q}}$, ${\rm rank}B_j=l(j)$, $1\leq j\leq r$, and $\varphi(s)$ belongs to the class ${\widetilde S}$. Let the function $F_j: {\mathbb{C}}^{N(\lambda+1)}\to {\mathbb{C}}$ be a continuous function for each $j=0,\ldots,n$. Suppose that the function
\begin{align*}
&\sum_{j=0}^{n}s^j \cdot F_j\biggl(
\varphi(s),\varphi'(s),\ldots,\varphi^{(N-1)}(s),\\
&\qquad \qquad
\zeta(s,\alpha_1;{\mathfrak{B}}_{11}),\zeta'(s,\alpha_1;{\mathfrak{B}}_{11}),\ldots,\zeta^{(N-1)}(s,\alpha_1;{\mathfrak{B}}_{11}),\ldots,\\
& \qquad\qquad
\zeta(s,\alpha_1;{\mathfrak{B}}_{1l(1)}),\zeta'(s,\alpha_1;{\mathfrak{B}}_{1l(1)}),\ldots,\zeta^{(N-1)}(s,\alpha_1;{\mathfrak{B}}_{1l(1)}),\ldots,\\
& \qquad\qquad
\zeta(s,\alpha_r;{\mathfrak{B}}_{r1}),\zeta'(s,\alpha_r;{\mathfrak{B}}_{r1}),\ldots,\zeta^{(N-1)}(s,\alpha_r;{\mathfrak{B}}_{r1}),\ldots,\\
& \qquad\qquad
\zeta(s,\alpha_r;{\mathfrak{B}}_{rl(r)}),\zeta'(s,\alpha_r;{\mathfrak{B}}_{rl(r)}),\ldots,\zeta^{(N-1)}(s,\alpha_r;{\mathfrak{B}}_{rl(r)})
\biggr)
\end{align*}
is identically zero. Then $F_j\equiv 0$ for $j=0,\ldots,n$.
\end{theorem}

This theorem is a consequence of the following mixed joint universality theorem, which is
a generalization of Lemma \ref{km-lem-1}.
This theorem is also an analogue of a result of J.~Genys, R.~Macaitien{\.e},
S.~Ra{\v c}kauskien{\.e} and D.~{\v S}iau{\v c}i{\=u}nas
\cite[Theorem 3]{JG-RM-SR-DS-10}, which treats the case that $\varphi(s)$ is replaced by
$\zeta(s)$.

\begin{theorem}\label{km-thm-3}
Suppose that $\alpha_1,\ldots,\alpha_r$ are algebraically independent over ${\mathbb{Q}}$, ${\rm rank}B_j=l(j)$, $1\leq j\leq r$, and $\varphi(s)$ belongs to the class
${\widetilde S}$.
Let $K_1$ be a compact subset of $D(\sigma^*,1)$, $K_{2jl}$ be compact subsets of
$D(\frac{1}{2},1)$,
all of which with connected complements.
Suppose that $f_1\in H_0^c(K_1)$ and $f_{2jl}\in H^c(K_{2jl})$.
Then, for every $\varepsilon>0$,
\begin{align*}
\liminf\limits_{T \to \infty}\frac{1}{T}\mu\bigg\{\tau\in [0,T]: &\;
\sup\limits_{s \in K_1}|\varphi(s+i\tau)-f_1(s)|<\varepsilon, \\
& \;\max\limits_{1\leq j\leq r}\max\limits_{1\leq l\leq l(j)}
\sup\limits_{s\in K_{2jl}}|\zeta(s+i\tau,\alpha_j;{\mathfrak{B}}_{jl})-f_{2jl}(s)|<
\varepsilon\bigg\}>0.
\end{align*}
\end{theorem}

\begin{remark}
Consider the case when all $l(j)=1$, $1\leq j\leq r$.
Write $\mathfrak{B}_{j1}=\mathfrak{B}_j$.
In this case, the condition ${\rm rank}B_j=l(j)$ trivially holds.
Therefore the joint universality and the functional independence among the functions
$\varphi(s)$,
$\zeta(s,\alpha_1,\mathfrak{B}_1),\ldots,\zeta(s,\alpha_r,\mathfrak{B}_r)$ are valid
without any matrix condition, only under the assumption that
$\alpha_1,\ldots,\alpha_r$ are algebraically independent over ${\mathbb{Q}}$.
When A.~Laurin{\v c}ikas started his study on the universality of periodic Hurwitz
zeta-functions, he assumed various matrix-type conditions, but finally,
the joint universality
among periodic Hurwitz zeta-functions without any matrix condition was established
by A.~Laurin{\v c}ikas and S.~Skerstonait{\.e} in \cite[Theorem~3]{AL-SS-09a}.
Our Theo\-rems \ref{km-thm-2} and \ref{km-thm-3} include the ``mixed''
generalization of this theorem.
\end{remark}

\section{Proof of Theorems~\ref{km-thm-2} and \ref{km-thm-3}}

In the proof of Theorems~\ref{km-thm-2} and \ref{km-thm-3}, the crucial role is
played by a mixed joint limit theorem in the sense of weakly convergent probability measures in the space of analytic functions.

\subsection{A generalized mixed joint limit theorem}

Let $D_1$ be an open subset of $D(\sigma^*,1)$, and $D_2$ be an open subset of
$D(\frac{1}{2},1)$.
For any set $S$, by ${\mathcal{B}}(S)$ we denote the set of all Borel subsets of $S$.
For any region $D$, denote by $H(D)$ the set of all holomorphic functions on $D$.
Let
${\underline H}$ be the Cartesian product of $\lambda+1$ such spaces, i.e.
$$
{\underline H}=H(D_1)\times \underbrace{H(D_2)\times ...\times H(D_2)}\limits_{\lambda}.
$$
Moreover, let
$$
\Omega_1=\prod\limits_{p\in\mathbb{P}}\gamma_p \quad \text{and}\quad \Omega_2=\prod\limits_{m=0}^{\infty}\gamma_m,
$$
where $\gamma_p=\gamma$ for all $p\in \mathbb{P}$, $\gamma_m=\gamma$ for all
$m \in \mathbb{N}_0$, and $\gamma=\{s\in \mathbb{C}: |s|=1\}$,
and define
$$
{\underline \Omega}=\Omega_1\times \Omega_{21}\times ... \times \Omega_{2 r}
$$
where $\Omega_{2j}=\Omega_2$ for all $j=1,...,r$. Then by the Tikhonov theorem ${\underline \Omega}$ is a compact topological Abelian group also. Then we have the probability space $({\underline \Omega},{\mathcal B}({\underline \Omega}), {\underline m}_H)$. Here ${\underline m}_H$ is the  product of Haar measures $m_{H1}, m_{H21},...,m_{H2r}$,
where $m_{H1}$ is the probability Haar measure on $(\Omega_1,{\mathcal{B}}(\Omega_1))$ and
$m_{H2j}$ is the probability Haar measure on $(\Omega_{2j},{\mathcal{B}}(\Omega_{2j}))$, $j=1,...,r$. Let $\omega_1(p)$ be the projection of $\omega_1 \in \Omega_1$ to $\gamma_p$, and, for every $m \in \mathbb{N}$, define
$$
\omega_1(m)=\prod_{p^\alpha \| m}\omega_1(p)^\alpha
$$
with respect to the factorization of $m$ to primes. Denote by $\omega_{2j}(m)$ the projection of $\omega_{2j} \in \Omega_{2j}$ to $\gamma_m$, $m \in \nn_0$, $j=1,...,r$.

For brevity, write ${\underline \alpha}=\left(\alpha_1,\ldots,\alpha_{r}\right)$,
$\underline{\mathfrak{B}}=\left(\mathfrak{B}_{11},\ldots, \mathfrak{B}_{1l(1)},
\cdots,\mathfrak{B}_{r1},\cdots,\mathfrak{B}_{rl(r)}\right)$,
$\underline{s}=(s_1,s_{211},\ldots,s_{21l(1)},\ldots,s_{2r1},\ldots,s_{2rl(r)})
\in\mathbb{C}^{\lambda+1}$,
and
\begin{align*}
&\lefteqn{\underline{Z}(\underline{s},\underline{\alpha};\underline{\mathfrak{B}})=
(\varphi(s_1),\zeta(s_{211},\alpha_1;{\mathfrak{B}}_{11}),\ldots,
\zeta(s_{21l(1)},\alpha_1;{\mathfrak{B}}_{1l(1)}),}\\
&\qquad\ldots,
\zeta(s_{2r1},\alpha_r;{\mathfrak{B}}_{r1}),\ldots,
\zeta(s_{2rl(r)},\alpha_r;{\mathfrak{B}}_{rl(r)})
).
\end{align*}

Let ${\underline \omega}=\left(\omega_1,\omega_{21},...,\omega_{2r}\right)
\in{\underline \Omega}$.
Define the ${\underline H}$-valued random element
$\underline{Z}(\underline{s},\underline{\alpha},{\underline\omega};\underline{\mathfrak{B}})$ on the probability space $({\underline \Omega},{\mathcal{B}}({\underline \Omega}),{\underline m}_H)$ by the formula
\begin{align*}
&\underline{Z}(\underline{s},\underline{\alpha},{\underline \omega};\underline{\mathfrak{B}})=
(\varphi(s_1,\omega_1),\zeta(s_{211},\alpha_1,\omega_{21};\mathfrak{B}_{11}),\ldots,
\zeta(s_{21l(1)},\alpha_1,\omega_{21};\mathfrak{B}_{1l(1)}),\\
&\qquad\ldots,
\zeta(s_{2r1},\alpha_r,\omega_{2r};\mathfrak{B}_{r1}),\ldots,
\zeta(s_{2rl(r)},\alpha_{r},\omega_{2r};\mathfrak{B}_{rl(r)})),
\end{align*}
where
$$
\varphi(s,\omega_1)=\sum_{m=1}^{\infty}\frac{c_m\omega_1(m)}{{m^s}}, \quad s
\in D(\sigma^{*},1),
$$
and
$$
\zeta(s,\alpha_j,\omega_{2j};\mathfrak{B}_{jl})=\sum_{m=0}^{\infty}\frac{b_{mjl}\omega_{2j}(m)}{(m+\alpha_j)^s}, \quad s \in D\big(\frac{1}{2},1\big), \quad j=1,...,r, \quad l=1,...,l(r),
$$
respectively.    These series are convergent for almost all $\omega_1 \in \Omega_1$ and
$\omega_{2j} \in \Omega_{2j}$, $j=1,...,r$.
Denote by $P_{\underline{Z}}$ the distribution of the random element
$\underline{Z}(\underline{s},\underline{\alpha},{\underline \omega};
\underline{\mathfrak{B}})$, i.e.,
$$
P_{\underline{Z}}(A)={\underline m}_H\left({\underline \omega} \in {\underline \Omega}:
\underline{Z}(\underline{s},{\underline \alpha},{\underline\omega};{\underline{\mathfrak{B}}}) \in A\right), \quad A \in {\mathcal{B}}({\underline H}).
$$

Now we are ready to state our mixed joint limit theorem for the functions
$\varphi(s)$, $\zeta(s,\alpha_1;{\mathfrak{B}}_{11}),$ $\ldots,
\zeta(s,\alpha_r;{\mathfrak{B}}_{rl(r)})$ as a lemma.

\begin{lemma}\label{thm-rk}
Suppose that the numbers $\alpha_1,...,\alpha_{r}$ are algebraically independent over
$\mathbb{Q}$, 
and $\varphi\in {\widetilde S}$.
Then the probability measure $P_T$ defined by
\begin{align}\label{def-P_T}
P_T(A)=\frac{1}{T}{\rm meas}\left\{\tau \in [0,T]: \underline{Z}(\underline{s}+i\tau,\underline{\alpha};\underline{\mathfrak{B}})  \in A\right\},\quad A \in {\mathcal{B}}({\underline H}),
\end{align}
converges weakly to $P_{\underline{Z}}$ as $T\to \infty$.
\end{lemma}

When $r=1$ and $l(1)=1$, this lemma is Theorem 2.1 from \cite{RK-KM-15}. 
On the other hand, the same type of limit theorem with $\varphi$ replaced by
the Riemann zeta-function is given in \cite[Theorem 4]{JG-RM-SR-DS-10}.
The proof of the above lemma is quite similar to that of those results, so here we
give a very brief outline.

Let $\varphi_n(s)$, $\varphi_n(s,\widehat{\omega}_1)$,
$\zeta_n(s,\alpha;\mathfrak{B})$, $\zeta_n(s,\alpha,\widehat{\omega}_2;\mathfrak{B})$
be the same as in \cite[Section 3]{RK-KM-15}, and define
\begin{align*}
&\lefteqn{\underline{Z}_n(\underline{s},\underline{\alpha};\underline{\mathfrak{B}})=
(\varphi_n(s_1),\zeta_n(s_{211},\alpha_1;{\mathfrak{B}}_{11}),\ldots,
\zeta_n(s_{21l(1)},\alpha_1;{\mathfrak{B}}_{1l(1)}),}\\
&\qquad\qquad \qquad\ldots,
\zeta_n(s_{2r1},\alpha_r;{\mathfrak{B}}_{r1}),\ldots,
\zeta_n(s_{2rl(r)},\alpha_r;{\mathfrak{B}}_{rl(r)}))
\end{align*}
and
\begin{align*}
&\underline{Z}_n(\underline{s},\underline{\alpha},{\underline{ \widehat\omega}};\underline{\mathfrak{B}})=
(\varphi_n(s_1,\widehat{\omega}_1),\zeta_n(s_{211},\alpha_1,\widehat{\omega}_{21};
\mathfrak{B}_{11}),\ldots,
\zeta_n(s_{21l(1)},\alpha_1,\widehat{\omega}_{21};\mathfrak{B}_{1l(1)}),\\
&\qquad\qquad \qquad\quad \ldots,
\zeta_n(s_{2r1},\alpha_r,\widehat{\omega}_{2r};\mathfrak{B}_{r1}),\ldots,
\zeta_n(s_{2rl(r)},\alpha_{r},\widehat{\omega}_{2r};\mathfrak{B}_{rl(r)})),
\end{align*}
where ${\underline{ \widehat\omega}}=(\widehat{\omega}_1,\widehat{\omega}_{21},
\ldots,\widehat{\omega}_{2r})\in\underline{\Omega}$.
The probability measures $P_{T,n}$, $\widehat{P}_{T,n}$ and $\widehat{P}_T$ on
$\underline{H}$ are defined similarly as
$P_T$ in \eqref{def-P_T}, with replacing
$\underline{Z}(\underline{s}+i\tau,\underline{\alpha};\underline{\mathfrak{B}})$ by
$\underline{Z}_n(\underline{s}+i\tau,\underline{\alpha};\underline{\mathfrak{B}})$,
$\underline{Z}_n(\underline{s}+i\tau,\underline{\alpha},{\underline{ \widehat\omega}};
\underline{\mathfrak{B}})$ and
$\underline{Z}(\underline{s}+i\tau,\underline{\alpha},{\underline{ \widehat\omega}};
\underline{\mathfrak{B}})$,
respectively.
As an analogue of \cite[Lemma~3]{RK-KM-15} or \cite[Lemma~2]{JG-RM-SR-DS-10}, we can
show that both measures $P_{T,n}$ and $\widehat{P}_{T,n}$ converge weakly to a certain
probability measure $P_n$ as $T\to\infty$.
(In the course of the proof, a key is Lemma~1 from \cite{JG-RM-SR-DS-10}, which is based
on the assumption that $\alpha_1,...,\alpha_{r}$ are algebraically independent over
$\mathbb{Q}$.)

Next we need an approximation lemma in the mean value sense, similar to
\cite[Lem\-ma~3.3]{RK-KM-15}.
In the proof of \cite[Lem\-ma~3.3]{RK-KM-15}, we used the result of \cite{AJ-AL-06}.
The desired approxi\-mation can be shown by using, instead of \cite{AJ-AL-06}, mean value
results given in \cite[(2.3), (2.5)]{AL-SS-09b}.

Then we can prove that both $P_T$ and $\widehat{P}_T$ converge weakly to a certain
probability measure $P$.    This is an analogue of \cite[Lemma~3.4]{RK-KM-15} or
\cite[Lemma~5]{JG-RM-SR-DS-10}.
Finally we can show that $P=P_{\underline{Z}}$, by the usual ergodic argument.
This completes the proof of Lemma \ref{thm-rk}.

\subsection{Completion of the proof}

Now we complete the proof of Theorems \ref{km-thm-2} and \ref{km-thm-3}.
Hereafter we assume that $\alpha_1,\ldots,\alpha_r$ are algebraically independent over
${\mathbb{Q}}$, ${\rm rank}B_j=l(j)$, $1\leq j\leq r$, and $\varphi(s)$ belongs to
the class ${\widetilde S}$.

Let $S_{\varphi}$ be the set of all $f\in H(D_1)$ which is non-vanishing on $D_1$,
or constantly $\equiv 0$ on $D_1$.

\begin{lemma}\label{lem-support}
The support of the measure $P_{\underline{Z}}$ is $S_{\varphi}\times H(D_2)^{\lambda}$.
\end{lemma}

This can be shown just analogously to \cite[Theorem 5]{JG-RM-SR-DS-10}.   Then
Theorem \ref{km-thm-3} follows from Lemmas \ref{thm-rk} and \ref{lem-support}
by the standard argument; see again \cite{JG-RM-SR-DS-10}.

Next, as a generalization of Lemma \ref{km-lem-2}, we can show
that the image of the map
$u: {\mathbb{R}} \to {\mathbb{C}}^{(\lambda+1)N}$ by the formula
\begin{align*}
&u(t)=\biggl(\varphi(\sigma+it),\varphi'(\sigma+it),\ldots,\varphi^{(N-1)}(\sigma+it),\\
& \zeta(\sigma+it,\alpha_1;{\mathfrak{B}}_{11}),\zeta'(\sigma+it,\alpha_1;
{\mathfrak{B}}_{11}),\ldots,\zeta^{(N-1)}(\sigma+it,\alpha_1;{\mathfrak{B}}_{11}),\ldots,\\
&\zeta(\sigma+it,\alpha_1;{\mathfrak{B}}_{1l(1)}),\zeta'(\sigma+it,\alpha_1;
{\mathfrak{B}}_{1l(1)}),\ldots,\zeta^{(N-1)}(\sigma+it,\alpha_1;{\mathfrak{B}}_{1l(1)}),
\ldots,\\
&\zeta(\sigma+it,\alpha_r;{\mathfrak{B}}_{r1}),\zeta'(\sigma+it,\alpha_r;
{\mathfrak{B}}_{r1}),\ldots,\zeta^{(N-1)}(\sigma+it,\alpha_r;{\mathfrak{B}}_{r1}),\ldots,\\
&\zeta(\sigma+it,\alpha_r;{\mathfrak{B}}_{rl(r)}),\zeta'(\sigma+it,\alpha_r;
{\mathfrak{B}}_{rl(r)}),\ldots,\zeta^{(N-1)}(\sigma+it,\alpha_r;{\mathfrak{B}}_{rl(r)})
\biggr)
\end{align*}
is dense in $\mathbb{C}^{(\lambda+1)N}$.    Using this denseness result, similar to
the proof of Theorem \ref{km-thm-1}, we obtain the assertion of
Theorem \ref{km-thm-2}.

\section*{Appendix}

Here we give a comment on \cite{KM-90}.    On p.179 of \cite{KM-90}, it is mentioned
that the Euler product \eqref{km-eq-1} is convergent absolutely, under the condition
\eqref{km-eq-2}, in the region $\sigma>\alpha+\beta+1$.   This can be seen by
the estimate
\begin{align*}
\sum_{m=1}^{\infty}\sum_{j=1}^{g(m)}|a_m^{(j)}|p_m^{-\sigma f(j,m)}
\leq \sum_{m=1}^{\infty}\sum_{j=1}^{g(m)}p_m^{\beta-\sigma}
\ll \sum_{m=1}^{\infty}p_m^{\alpha+\beta-\sigma},
\end{align*}
and hence the Dirichlet series expansion \eqref{km-eq-2.5} is also valid in that
region.

In the same page of \cite{KM-90}, the estimate
$\widetilde{c}_k=O(k^{\alpha+\beta})$ is also stated.    However this is to be
amended as follows.   From \eqref{km-eq-1} we have
$$
\widetilde{\varphi}(s)=\prod_{m=1}^{\infty}\sum_{l=0}^{\infty}
\left({\sum}^*(a_m^{(1)})^{h_1}\cdots(a_m^{(g(m))})^{h_{g(m)}}\right)p_m^{-ls},
$$
where ${\sum}^*$ means the summation over all tuples $(h_1,\ldots,h_{g(m)})$ of
non-negative integers satisfying
$$
h_1f(1,m)+\ldots+h_{g(m)}f(g(m),m)=l.
$$
Denote by $C(m,l)$ the number of such tuples.   Using \eqref{km-eq-2} we obtain
$$
{\sum}^*(a_m^{(1)})^{h_1}\cdots(a_m^{(g(m))})^{h_{g(m)}}
\leq {\sum}^* p_m^{(h_1+\ldots+h_{g(m)})\beta}
\leq p_m^{l\beta}C(m,l).
$$
To estimate $C(m,l)$, it suffices to consider the case when
$f(1,m)=\ldots=f(g(m),m)=1$, and in this case
$$
C(m,l)=\binom{g(m)+l-1}{l}\leq g(m)^l\leq (C_1 p_m^{\alpha})^l.
$$
Therefore ${\sum}^*\leq C_1^l p_m^{(\alpha+\beta)l}$, which yields, if
$k=p_1^{l_1}\ldots p_r^{l_r}$,
$$
\widetilde{c}_k\leq C_1^{l_1+\ldots+l_r}(p_1^{l_1}\ldots p_r^{l_r})^{\alpha+\beta}
=C_1^{\Omega(k)}k^{\alpha+\beta},
$$
where $\Omega(k)$ denotes the total number of prime divisors of $k$.

For any $\varepsilon>0$, we see that
$C_1 p_m^{\alpha}\leq p_m^{\alpha+\varepsilon}$ if $m$ is sufficiently large,
and then ${\sum}^*\leq (p_m^{\alpha+\beta+\varepsilon})^l$.
This implies $\widetilde{c}_k=O(k^{\alpha+\beta+\varepsilon})$, if all prime factors
of $k$ are large.


\end{document}